\documentclass[11pt,letterpaper]{amsart}

\usepackage{amssymb}

\newtheorem{theorem}{Theorem}[section]
\newtheorem{proposition}[theorem]{Proposition}
\newtheorem{corollary}[theorem]{Corollary}
\newtheorem{lemma}[theorem]{Lemma}
\newtheorem*{claim}{Claim}

\theoremstyle{definition}
\newtheorem{definition}[theorem]{Definition}

\newcommand{\F}{\mathbb{F}}
\renewcommand{\P}{\mathbb{P}}
\newcommand{\A}{\mathbb{A}}

\begin{document}

\title{On the classification of hyperovals}

\author{Florian Caullery}
\address{Institut de Math{\'e}matiques de Luminy, CNRS-UPR9016, 163 av. de Luminy, case 907, 13288 Marseille Cedex 9, France.}
\email[F. Caullery]{florian.caullery@etu.univ-amu.fr}

\author{Kai-Uwe Schmidt}
\address{Faculty of Mathematics, Otto-von-Guericke University, Universit\"atsplatz~2, 39106 Magdeburg, Germany.}
\email[K.-U. Schmidt]{kaiuwe.schmidt@ovgu.de}

\date{11 March 2014 (revised 29 May 2014)}



\begin{abstract}
A hyperoval in the projective plane $\P^2(\F_q)$ is a set of $q+2$ points no three of which are collinear. Hyperovals have been studied extensively since the 1950s with the ultimate goal of establishing a complete classification. It is well known that hyperovals in $\P^2(\F_q)$ are in one-to-one correspondence to polynomials with certain properties, called o-polynomials of $\F_q$. We classify o-polynomials of $\F_q$ of degree less than $\frac12q^{1/4}$. As a corollary we obtain a complete classification of exceptional o-polynomials, namely polynomials over $\F_q$ that are o-polynomials of infinitely many extensions of $\F_q$.
\end{abstract}

\maketitle

\section{Introduction and results}

An \emph{arc} in the projective plane $\P^2(\F_q)$ is a set of points of $\P^2(\F_q)$ no three of which are collinear. It is well known that the maximum number of points in an arc in $\P^2(\F_q)$ is $q+1$ for odd $q$ and $q+2$ for even $q$ (see~\cite[Chapter 8]{Hir1998}, for example). Accordingly, an arc of size $q+1$ is called an \emph{oval} and an arc of size $q+2$ is called a \emph{hyperoval}. By a theorem due to Segre~\cite{Seg1955}, every oval in $\P^2(\F_q)$ of odd order is a conic, which at once classifies ovals in $\P^2(\F_q)$ for odd~$q$. In contrast, the classification of hyperovals is a major open problem in finite geometry, which has attracted sustained interest over the last sixty years. For surveys on progress toward this classification we refer to~\cite{Che1996},~\cite{Hir1998}, and~\cite{Pen2003}, for example.
\par
Throughout this paper we let $q$ be a power of two. Hyperovals have a canonical description via polynomials over $\F_q$. 
\begin{definition}
An \emph{o-polynomial} of $\F_q$ is a polynomial $f\in\F_q[x]$ of degree at most $q-1$ that induces a permutation on $\F_q$ satisfying $f(0)=0$ and $f(1)=1$ and
\begin{equation}
\det
\begin{pmatrix}
1 & 1 & 1\\
a & b & c\\
f(a) & f(b) & f(c)
\end{pmatrix}
\ne 0 \quad\text{for all distinct $a,b,c\in\F_q$}.   \label{eqn:det}
\end{equation}
\end{definition}
\par
By a suitable choice of coordinates, we may assume without loss of generality that the points $(1,0,0)$, $(0,1,0)$, $(0,0,1)$, $(1,1,1)$ are contained in a hyperoval. It is well known (and easily verified) that every such hyperoval in $\P^2(\F_q)$ can be written as
\begin{equation}
\{(f(c),c,1):c\in\F_q\}\cup\{(1,0,0),(0,1,0)\},   \label{eqn:hyperoval}
\end{equation}
where $f$ is an o-polynomial of $\F_q$. Conversely, if $f$ is an o-polynomial of $\F_q$, then~\eqref{eqn:hyperoval} is a hyperoval in $\P^2(\F_q)$.
\par
For example, $f(x)=x^2$ is an o-polynomial of $\F_{2^h}$ for all $h>1$. There exist several other infinite families of o-polynomials and some sporadic examples. For a list of known hyperovals, as of 2003, we refer to~\cite{Pen2003}. Since 2003, no new hyperovals have been found.
\par
O-polynomials of $\F_{2^h}$ have been classified for $h\le 5$~\cite{Hal1975},~\cite{OKePen1991},~\cite{PenRoy1994} and monomial o-polynomials of $\F_{2^h}$ have been classified for $h\le 30$~\cite{Gly1989}. There is also a classification of monomial o-polynomials of a certain form, namely those of degree $2^i+2^j$~\cite{CheSto1998} or $2^i+2^j+2^k$~\cite{Vis2010}. O-polynomials of degree at most $6$ are also classified~\cite[Theorem~8.31]{Hir1998}.
\par
Our main result is the following classification of low-degree o-polynomials. Call two polynomials $f,g\in\F_q[x]$ with $f(0)=g(0)=0$ and $f(1)=g(1)=1$ \emph{equivalent} if there exists an $a\in\F_q$ such that
\[
g(x)=\frac{f(x+a)+f(a)}{f(1+a)+f(a)}.
\]
It is readily verified that this equivalence indeed defines an equivalence relation and that it preserves the property of being an o-polynomial of $\F_q$.
\begin{theorem}
\label{thm:main}
If $f$ is an o-polynomial of $\F_q$ of degree less than $\frac{1}{2}q^{1/4}$, then~$f$ is equivalent to either $x^6$ or $x^{2^k}$ for a positive integer~$k$.  
\end{theorem}
\par
It is well known that $x^6$ is an o-polynomial of $\F_{2^h}$ if and only if $h$ is odd and that $x^{2^k}$ is an o-polynomial of $\F_{2^h}$ if and only if $k$ and~$h$ are coprime.
\par
Now consider polynomials $f\in\F_q[x]$ with the property that $f$ is an o-polynomial of $\F_{q^r}$ for infinitely many $r$; we call such a polynomial an \emph{exceptional} o-polynomial of $\F_q$. Exceptional o-polynomials provide a uniform construction for hyperovals in infinitely many projective planes. Theorem~\ref{thm:main} gives a complete classification of exceptional o-polynomials.
\begin{corollary}
\label{cor:exceptional}
If $f$ is an exceptional o-polynomial of $\F_q$, then $f$ is equivalent to either $x^6$ or $x^{2^k}$ for a positive integer~$k$.  
\end{corollary}
\par
The specialisation of Corollary~\ref{cor:exceptional} to the case that $f$ is a monomial was conjectured by Segre and Bartocci~\cite{SegBar1971} and was recently proved by Hernando and McGuire~\cite{HerMcG2012} (another, much simpler, proof of this case was later given by Zieve~\cite{Zie2013}).

\section{Proof of Theorem~\ref{thm:main}}

We begin with recalling several standard results, for which proofs can be found in~\cite[Chapter~8]{Hir1998}, for example. Our first result is an almost immediate consequence of the definition of an o-polynomial.
\begin{lemma}[{\cite[Corollary 8.23]{Hir1998}}]
\label{lem:odd_powers}
Every o-polynomial of $\F_q$ with $q>2$ has only terms of positive even degree.
\end{lemma}
\par
We need the following (easy) classification of o-polynomials of degree~$6$.
\begin{lemma}[{\cite[Theorem 8.31]{Hir1998}}]
\label{lem:degree_6}
If $f$ is an o-polynomial of degree~$6$, then~$f$ is equivalent to $x^6$.
\end{lemma}
\par
We also need the following result, originally proved by Payne~\cite{Pay1971} and later by Hirschfeld~\cite{Hir1975} with a different method, classifying translation hyperovals.
\begin{lemma}[{\cite[Theorem 8.41]{Hir1998}}]
\label{lem:lin_o_poly}
Every o-polynomial, in which the degree of every term is a power of two, is in fact a monomial.
\end{lemma}
\par
Now let $f\in\F_q[x]$ and define the polynomial
\begin{align*}
\Phi_f(x,y,z)&=\frac{1}{(x+y)(x+z)(y+z)}\cdot \det
\begin{pmatrix}
1 & 1 & 1\\
x & y & z\\
f(x) & f(y) & f(z)
\end{pmatrix}\\[2ex]
&=\frac{x(f(y)+f(z))+y(f(x)+f(z))+z(f(x)+f(y))}{(x+y)(x+z)(y+z)}.
\end{align*}
The condition~\eqref{eqn:det} is equivalent to the condition that all points in $\A^3(\F_q)$ of the surface defined by
\[
\Phi_f(x,y,z)=0
\]
satisfy $x=y$, $x=z$, or $y=z$. This leads us to the following result, which essentially follows from a refinement of the Lang-Weil bound~\cite{LanWei1954} for the number of $\F_q$-rational points in algebraic varieties.
\begin{proposition}
\label{pro:Weil}
Let $f\in\F_q[x]$ be of degree less than $\frac12q^{1/4}$. If $\Phi_f$ has an absolutely irreducible factor over $\F_q$, then $f$ is not an o-polynomial of $\F_q$. 
\end{proposition}
\begin{proof}
If $f$ has degree $0$ or $1$, then $f$ is not an o-polynomial by Lemma~\ref{lem:odd_powers}, so assume that $f$ has degree at least $2$. We first show that $\Phi_f$ is not divisible by $x+y$, $x+z$, or $y+z$. Suppose, for a contradiction, that $\Phi_f$ is divisible by $x+y$. Then the partial derivative of 
\[
x(f(y)+f(z))+y(f(x)+f(z))+z(f(x)+f(y))
\]
with respect to $x$ is divisible by $x+y$, or equivalently,
\[
f(y)+f(z)+(y+z)f'(y)=0.
\]
This forces the degree of $f$ to be $0$ or $1$, contradicting our assumption. Hence, by symmetry, $\Phi_f$ is not divisible by $x+y$, $x+z$, or $y+z$. Therefore, $\Phi_f(x,y,x)$, $\Phi_f(x,y,y)$, and $\Phi_f(x,x,z)$ are nonzero polynomials, and so each has at most $d\,q$ zeros in $\A^2(\F_q)$, where $d$ is the degree of $\Phi_f$ (see~\cite[Theorem~6.13]{LidNie1997}, for example).
\par
Now suppose that $\Phi_f$ has an absolutely irreducible factor over $\F_q$. Then, by a refinement of the Lang-Weil bound~\cite{LanWei1954} due to Ghorpade and Lachaud~\cite[11.3]{GhoLau2002}, the number of points in $\A^3(\F_q)$ of the surface defined by $\Phi_f(x,y,z)=0$ is at least
\[
q^2-(d-1)(d-2)q^{3/2}-12(d+3)^4\,q.
\]
Hence the number of such points that are not on one of the planes $x=y$, $x=z$, or $y=z$ is at least
\[
q^2-(d-1)(d-2)q^{3/2}-12(d+3)^4\,q-3dq,
\]
which is positive since
\[
0\le d\le\tfrac12q^{1/4}-3.
\]
Then our remarks preceding the proposition imply that $f$ is not an o-polynomial of $\F_q$.
\end{proof}
\par
In order to prove Theorem~\ref{thm:main}, we first use the constraints given by Lemmas~\ref{lem:odd_powers},~\ref{lem:degree_6}, and~\ref{lem:lin_o_poly} and then show that in all remaining cases, $\Phi_f$ has an absolutely irreducible factor over $\F_q$ unless $f$ is one of the polynomials in Theorem~\ref{thm:main}. To do so, we frequently use the polynomials
\begin{equation}
\phi_j(x,y,z)=\frac{x(y^j+z^j)+y(x^j+z^j)+z(x^j+y^j)}{(x+y)(x+z)(y+z)}.   \label{def_phi}
\end{equation}
Then, writing 
\[
f(x)=\sum_{i=0}^da_ix^i,
\]
we have
\[
\Phi_f(x,y,z)=\sum_{i=0}^da_i\phi_i(x,y,z).
\]
\par
If $j$ is an even positive integer, not equal to $6$ or a power of two, then $\phi_j$ has an absolutely irreducible factor over $\F_2$ (and so proves Corollary~\ref{cor:exceptional} in the case that $f$ is a monomial). This was conjectured by Segre and Bartocci~\cite{SegBar1971} and proved by Hernando and McGuire~\cite{HerMcG2012} (and can also be deduced with a few extra steps from an argument due to Zieve~\cite[Section 5]{Zie2013}).
\begin{lemma}[{\cite[Theorem 8]{HerMcG2012}}]
\label{lem:phi_j_abs_irreducible}
Let $j$ be an even positive integer, not equal to $6$ or a power of two. Then $\phi_j$ has an absolutely irreducible factor over~$\F_2$.
\end{lemma}
\par
If $f$ is an o-polynomial of $\F_q$, then either $q=2$ and $f$ has degree~$1$ or $q>2$ and $f$ has positive even degree by Lemma~\ref{lem:odd_powers}. Hence to prove Theorem~\ref{thm:main}, we can assume that $f$ has positive even degree. In the case that $f$ has positive even degree that is neither $6$ nor a power of two, we show that $\Phi_f$ has an absolutely irreducible factor over $\F_q$, and using Proposition~\ref{pro:Weil} prove the statement of Theorem~\ref{thm:main} in this case.
\begin{proposition}
\label{pro:6_or_2k}
Let $f\in\F_q[x]$ be of positive even degree not equal to $6$ or a power of two. Then $\Phi_f$ has an absolutely irreducible factor over $\F_q$.
\end{proposition}
\par
Proposition~\ref{pro:6_or_2k} will follow from Lemma~\ref{lem:phi_j_abs_irreducible} and the following simple observation due to Aubry, McGuire, and Rodier~\cite{AubMcGRod2010} (in which $\overline{\F}_q$ is the algebraic closure of $\F_q$).
\begin{lemma}[{\cite[Lemma~2.1]{AubMcGRod2010}}]
\label{lem:AMR}
Let $S$ and $P$ be projective surfaces in $\P^3(\overline{\F}_q)$ defined over $\F_q$. If $S\cap P$ has a reduced absolutely irreducible component defined over $\F_q$, then $S$ has an absolutely irreducible component defined~over~$\F_q$.
\end{lemma}
\par
\begin{proof}[Proof of Proposition~\ref{pro:6_or_2k}]
Write
\[
f(x)=\sum_{i=0}^da_ix^i,
\]
where $a_d\ne 0$, and consider the homogenisation of $\Phi_f$, namely
\[
\widetilde{\Phi}_f(w,x,y,z)=\sum_{i=0}^da_i\phi_i(x,y,z)\,w^{d-i}.
\]
The intersection of the projective surface defined by $\widetilde{\Phi}_f(w,x,y,z)=0$ with the plane defined by $w=0$ is the projective curve defined by $\phi_d(x,y,z)=0$ and $w=0$. By Lemma~\ref{lem:phi_j_abs_irreducible}, $\phi_d$ has an absolutely irreducible factor over $\F_q$. Notice that $\phi_d$ is square-free, which follows from the fact that the partial derivative of
\[
x(y^d+z^d)+y(x^d+z^d)+z(x^d+y^d)
\]
with respect to $x$ is in $\F_2[y,z]$ (using that $d$ is even) and from symmetry. Therefore, Lemma~\ref{lem:AMR} implies that $\widetilde{\Phi}_f$ (and therefore $\Phi_f$) has an absolutely irreducible factor over $\F_q$.
\end{proof}
\par
In view of Proposition~\ref{pro:6_or_2k} and Lemma~\ref{lem:degree_6}, it remains to prove Theorem~\ref{thm:main} when the degree of $f$ is a power of two. In view of Lemmas~\ref{lem:odd_powers} and~\ref{lem:lin_o_poly}, this case follows from Proposition~\ref{pro:Weil} and the following result.
\begin{proposition}
\label{pro:f_power_of_2}
Let $k$ be an integer satisfying $k\ge 2$ and let $f\in\F_q[x]$ be a polynomial of the form
\[
f(x)=\sum_{i=1}^{2^{k-1}}a_{2i}\,x^{2i}
\]
such that $a_{2^k}\ne 0$ and such that the degree of at least one term in $f$ is not a power of two. Then $\Phi_f$ is absolutely irreducible. 
\end{proposition}
\par
To prove Proposition~\ref{pro:f_power_of_2}, we use the following corollary to Lucas's theorem (see~\cite{Fin1947}, for example).
\begin{lemma}
\label{lem:Lucas}
The binomial coefficient $\tbinom{n}{m}$ is even if and only if at least one of the base-$2$ digits of $m$ is greater than the corresponding digit of $n$.
\end{lemma}
\par
\begin{proof}[Proof of Proposition~\ref{pro:f_power_of_2}]
Suppose, for a contradiction, that $\Phi_f$ is not absolutely irreducible. Let $\phi_j$ be defined by~\eqref{def_phi}. Our proof relies on the following claim. 
\begin{claim}
There exists $\theta\in\F_{2^k}-\F_2$ such that for all $i\in\{1,2,\dots,2^{k-1}\}$, we have
\[
a_{2i}=0\quad\text{or}\quad \text{$x+z+\theta (y+z)$ divides $\phi_{2i}(x,y,z)$}. 
\]
\end{claim}
\par
We defer the proof of the claim and first deduce the statement in the proposition from the claim. Let $n$ be an even integer such that $a_n$ is nonzero. By putting $x=\theta y+(\theta+1)z$ into
\[
(x+y)(x+z)(y+z)\phi_n(x,y,z),
\]
we see from the claim that
\[
yz^n+zy^n+(y^n+z^n)(\theta y+(\theta+1)z)+(y+z)(\theta y+(\theta+1)z)^n=0,
\]
which implies that
\[
(\theta+\theta^n)y^n+((\theta+1)+(\theta+1)^n)z^n\\
+\sum_{m=1}^{n-1}\binom{n}{m}\theta^my^m(\theta+1)^{n-m}z^{n-m}=0.
\]
Comparing coefficients, we find that $\binom{n}{m}$ is even for each $m\in\{1,\dots,n-1\}$. It is then readily verified that Lemma~\ref{lem:Lucas} implies that $n$ must be a power of two. Therefore, the degree of every term in $f$ is a power of two, contradicting our assumption. Hence $\Phi_f$ is absolutely irreducible.
\par
To prove the claim, we repeatedly use the identity
\begin{equation}
\phi_{2i}(x,y,x)=\bigg(\frac{x^i+y^i}{x+y}\bigg)^2\quad\text{for each $i\ge 1$},   \label{eqn:phi_i_xyx}
\end{equation}
which is elementary to verify. We also use the expansion
\[
\Phi_f=a_2\phi_2+a_4\phi_4+\cdots+a_{2^k}\phi_{2^k}.
\]
Since $\Phi_f$ is not absolutely irreducible by assumption, we may write
\begin{equation}
a_2\phi_2+a_4\phi_4+\cdots+a_{2^k}\phi_{2^k}=(P_s+P_{s-1}+\cdots+P_0)(Q_t+Q_{t-1}+\cdots+Q_0),   \label{eqn:F_expansion_2k}
\end{equation}
where $P_i$ and $Q_i$ are zero or homogeneous polynomials of degree $i$, defined over the algebraic closure of $\F_q$, and $s,t>0$ and $P_sQ_t$ is nonzero. Without loss of generality we may also assume that $s\le t$. We have
\begin{align}
\phi_{2^k}(x,y,z)&=\frac{(x+z)^{2^k-1}+(y+z)^{2^k-1}}{x+y}   \nonumber\\[1.5ex]
&=\prod_{\alpha\in\F_{2^k}-\F_2}\big(x+z+\alpha (y+z)\big).   \label{eqn:phi_2k_factorisation}
\end{align}
Since $a_{2^k}\phi_{2^k}=P_sQ_t$ by~\eqref{eqn:F_expansion_2k}, we find from~\eqref{eqn:phi_2k_factorisation} that $P_s$ and $Q_t$ are coprime and from~\eqref{eqn:phi_i_xyx} that
\begin{equation}
P_s(x,y,x)Q_t(x,y,x)=a_{2^k}(x+y)^{2^k-2}.   \label{eqn:PsQt_xyx}
\end{equation}
From~\eqref{eqn:F_expansion_2k} we have
\[
0=P_sQ_{t-1}+P_{s-1}Q_t.
\]
Since $P_s$ and $Q_t$ are coprime, we find that $P_s\mid P_{s-1}$, thus $P_{s-1}=0$ by a degree argument. Let $I$ be the smallest positive integer $i$ such that $a_{2^k-2i}$ is nonzero (this $I$ exists and satisfies $I<2^{k-1}$ by our assumed form of $f$). With a simple induction, involving the preceding argument, we conclude that
\begin{equation}
P_{s-1}=\dots=P_{s-2I+1}=0.   \label{eqn:PQ_zero}
\end{equation} 
In the next step we have from~\eqref{eqn:F_expansion_2k} that
\[
a_{2^k-2I}\phi_{2^k-2I}=P_sQ_{t-2I}+P_{s-2I}Q_t,
\]
which using~\eqref{eqn:PsQt_xyx} gives
\begin{multline}
a_{2^k-2I}\phi_{2^k-2I}(x,y,x)\\
=\beta a_{2^k}(x+y)^sQ_{t-2I}(x,y,x)+\beta^{-1}(x+y)^tP_{s-2I}(x,y,x)   \label{eqn:eqn:2k-2I}
\end{multline}
for some nonzero $\beta$ in the algebraic closure of $\F_q$. Write $I=2^\ell e$ for some nonnegative integer $\ell$ and some positive odd integer $e$. Using~\eqref{eqn:phi_i_xyx}, we have
\[
\phi_{2^k-2I}(x,y,x)=\Bigg(\frac{\big(x^{2^{k-\ell-1}-e}+y^{2^{k-\ell-1}-e}\big)^{2^\ell}}{x+y}\Bigg)^2.
\]
Since $2^{k-\ell-1}-e$ is odd, the polynomial
\[
x^{2^{k-\ell-1}-e}+y^{2^{k-\ell-1}-e}
\]
splits into distinct factors, and therefore the largest power of $x+y$ dividing $\phi_{2^k-2I}(x,y,x)$ is at most~$2(2^\ell-1)$. Hence, since $a_{2^k-2I}\ne 0$ and $s\le t$ by assumption, we have in view of~\eqref{eqn:eqn:2k-2I} that
\[
s\le 2(2^\ell-1)\le 2(I-1).
\]
Therefore, we find from~\eqref{eqn:PQ_zero} that $P_i=0$ unless $i=s$ and then from~\eqref{eqn:F_expansion_2k} that
\[
a_{2^k-2j}\phi_{2^k-2j}=P_sQ_{t-2j}\quad\text{for each $j\in\{0,1,\dots,2^{k-1}-1\}$}.
\]
This shows that $P_s$ divides $a_{2^k-2j}\phi_{2^k-2j}$ for each $j\in\{0,1,\dots,2^{k-1}-1\}$, which in view of $a_{2^k}\phi_{2^k}=P_sQ_t$ and~\eqref{eqn:phi_2k_factorisation} proves our claim.
\end{proof}


\providecommand{\bysame}{\leavevmode\hbox to3em{\hrulefill}\thinspace}
\providecommand{\MR}{\relax\ifhmode\unskip\space\fi MR }
\providecommand{\MRhref}[2]{%
  \href{http://www.ams.org/mathscinet-getitem?mr=#1}{#2}
}
\providecommand{\href}[2]{#2}


\end{document}